\documentclass[reqno]{amsart}
\usepackage[all]{xy}
\usepackage{amsmath}
\usepackage{amsfonts}
\usepackage{amssymb}
\usepackage{amsthm}
\usepackage{amscd}
\usepackage{latexsym}
\usepackage{txfonts}
\usepackage{tikz}
\usepackage{enumitem}
\usepackage{graphicx}
\usepackage[all]{xy}
\usepackage{fullpage}
\numberwithin{equation}{section}
\newcommand{\bb}{\mathbb}
\newcommand{\ov}{\overline}
\theoremstyle{plain}
\newtheorem{theorem}{Theorem}
\newtheorem{lemma}[theorem]{Lemma}
\newtheorem{corollary}[theorem]{Corollary}

\newtheorem*{theorem*}{Theorem}
\theoremstyle{remark} \newtheorem*{remark}{Remark}
\theoremstyle{remark} \newtheorem{example}{Example}
\DeclareMathOperator{\Ima}{im}
\newcommand{\mrm}{\mathrm}

\newcommand{\K}{K^{\times}}
\makeatletter
\def\paragraph{\@startsection{paragraph}{4}%
  \z@\z@{-\fontdimen2\font}%
  {\normalfont\bfseries}}
\makeatother

\begin{document}
\title{Critical groups of van Lint-Schrijver Cyclotomic Strongly Regular Graphs.}
\date{}
\author{Venkata Raghu Tej Pantangi}
\email{pvrt1990@ufl.edu}
\address{Department of Mathematics, University of Florida, Gainesville, FL 32611-8105, USA.}
\begin{abstract}
The \emph{critical} group of a finite connected graph is an abelian group defined by the Smith normal form of its Laplacian. 
Let $q$ be a power of a prime and $H$ be a multiplicative subgroup of $K=\bb{F}_{q}$. By $\mrm{Cay}(K,H)$ we denote the Cayley graph on the additive group of $K$ with ``connection'' set $H$. A strongly regular graph of the form $\mrm{Cay}(K,H)$ is called a \emph{cyclotomic strongly regular graph}. Let $\ell >2$ and $p$ be primes such that $p$ is primitive $\pmod{\ell}$. 
We compute the \emph{critical} groups of a family of \emph{cyclotomic strongly regular graphs} for which $q=p^{(\ell-1)t}$ (with $t\in \bb{N}$) and $H$ is the unique multiplicative subgroup of order $k=\frac{q-1}{\ell}$. These graphs were first discovered by van Lint and Schrijver in \cite{VS}.
\end{abstract}

\maketitle

\section{Introduction}
Let $\Gamma=(V,E)$ be a finite, simple, and connected graph. Let $A$ be the adjacency matrix of $\Gamma$ with respect to some arbitrary but fixed ordering of the vertex set $V$. Define the matrix $D$ to be the diagonal matrix of size $|V|$ whose $i$th diagonal entry is the valency of the $i$th vertex of $\Gamma$. The matrix $L:=D-A$ is called the Laplacian matrix of $\Gamma$. By $\bb{Z}^V$ we denote the free $\bb{Z}$-module with $V$ as a basis set. By abuse of notation, we may consider $L$ to be an element of $\mathrm{End}_{\bb{Z}}\left(\bb{Z}^V\right)$. The \emph{critical} group $C(\Gamma)$ is the finite part of the cokernal of $L$. 

This group is an invariants of $\Gamma$. By Kirchhoff's Matrix-tree theorem, it can be deduced that the order of $C(\Gamma)$ is equal to the number of spanning trees of $\Gamma$. (For instance, see \cite{sta}.) The critical groups of various graphs arise in graph theory in the context of the chip firing game (cf. \cite{Biggs}), as the \emph{abelian sandpile group} in statistical physics (cf. \cite{Dhar}), and also in arithmetic geometry (cf. \cite{Lor0}). Early works with computations of critical groups include \cite{Vince} and \cite{Dlor}. In \cite{Vince}, the critical groups of  Wheel graphs and complete bipartite graphs were computed. In the same paper, it was shown that the group depends only on the cycle matroid of the graph. The critical groups of complete bipartite graphs were computed independently in \cite{Dlor} as well.  

Other papers that include computation of critical groups of families of graphs include  \cite{Lor1}, \cite{CSX}, \cite{Bai}, \cite{JNR}, \cite{DS}, \cite{SDB}, and \cite{PS}. In \cite{Lor1}, Lorenzini examined the proportion of graphs with cyclic critical groups among graphs with critical groups of particular order. There are relatively few classes of graphs with known critical groups. A particular class of groups that has proved amenable to computations is the class of strongly regular graphs (for instance, see section 3 of\cite{Lor1}). In this paper we describe the critical groups of the \emph{cyclotomic strongly regular graphs} discovered in \cite{VS}. 

Consider a finite field $K$ of characteristic $p$ and a subgroup $H$ of $\K$. By $\mrm{Cay}(K,H)$ we denote the Cayley graph on the additive group of $K$ with ``connection'' set $H$. If $\mathrm{Cay}(K,H)$ is a strongly regular graph, then we speak of a \emph{cyclotomic strongly regular graph} (\emph{cyclotomic SRG}). The Paley graph is a well known example of a \emph{cyclotomic SRG}. This family of SRGs has been studied extensively by many authors; see \cite{VS, BWX, SW, FMX}. We refer the reader to section $4$ of \cite{QX} for a survey on these graphs. If $H$ is the multiplicative group of a non-trivial subfield of $K$, then $\mathrm{Cay}(K,H)$ is a \emph{cyclotomic SRG}. A graph of this form is called a \emph{subfield cyclotomic SRG}. Other examples of \emph{cyclotomic SRGs} are the \emph{semi-primitive cyclotomic SRGs}.
Consider a subgroup $H$ of $\K$ with $N:=[\K:H]>1$ and $N \mid \frac{|\K|}{p-1}$. 
Further assume that there exists an integer $s$ such that $p^{s} \equiv -1 \pmod{N}$. These arithmetic restrictions on $H$ ensure that the adjacency matrix of the regular graph $\mathrm{Cay}(K,H)$ has exactly $3$ eigenvalues and thus is a \emph{cyclotomic SRG} (see for example, Section 4 of \cite{QX}). A graph of this form is called a \emph{semi-primitive cyclotomic SRG}. According to a conjecture by Schmidt and White (Conjecture $4.4$ of \cite{SW}/ Conjecture $4.1$ of \cite{QX}), other than the above mentioned classes, there are only $11$ sporadic examples of \emph{cyclotomic SRGs}. In this paper we consider a class of \emph{semi-primitve cyclotomic SRGs} discovered in \cite{VS}.                        
  
Consider a pair of primes $(p,\ell)$ with $\ell \neq 2$, and a positive integer $t\in \bb{N}$. The graph $G(p,\ell,t)$ denotes $\mrm{Cay}(K,S)$, where $K=\bb{F}_{p^{(\ell-1)t}}$ and $S$ is the subgroup of index $\ell$ in $\K$ . Further assume (a) $p^{(\ell-1)t/2} \neq \ell-1$ whenever $t$ is odd; and (b) $p$ is primitive in $\bb{Z}/ \ell \bb{Z}$. The arithmetic constraint (a) is equivalent to the graph being connected, and (b) implies that $G(p,\ell,t)$ is a \emph{semi-primitive SRG}. These \emph{semi-primitive cylclotomic SRGs} were discovered in \cite{VS}. In this paper we describe the critical groups of this family of graphs. The construction of this family is similar to that of Paley and Peisert graphs. The critical group of the Paley graph was computed in \cite{CSX}, and that of Peisert graph was described in \cite{S}. We extend the techniques used in \cite{CSX} and \cite{S} to compute the critical group of $G(p,\ell,t)$ (with $(p,\ell,t)$ satisfying arithmetic constraints (a) and (b)).

We denote the critical group of $G(p,\ell,t)$ by $C$. Our results giving the elementary divisors of $C$ are stated in \S\ref{mainresult}. 
As $C$ is an abelian group, we can find subgroups $C_{p}$ (the Sylow $p$-subgroup of $C$),  $C_{p'}$ such that $C \cong C_{p}\oplus C_{p'}$ and $p \nmid |C_{p'}|$. We use different approaches to compute these two subgroups. Theorem \ref{pprime} describes $C_{p'}$. We apply a standard method of diagonalizing the Laplacian using the character table of $K=\bb{F}_{q}$ (here $q=p^{(\ell-1)t}$). A different approach is required to obtain a description of $C_{p}$. In \S\ref{block}, we study the permutation action of $S$ on $R$-free module $R^K$ with basis $K$, where $R$ is the ring of integers of a suitable extension of $\bb{Q}_{p}$. Let $\hat{S}$ be the group of $R$-valued characters of $S$. We obtain the decomposition $R^K=\bigoplus_{\chi \in \hat{S}}N_{\chi}$, where $N_{\chi}$ is the isotypic component of the $S$-module $R^K$ corresponding to the character $\chi$. Since $S$ preserves adjacency, each of these isotypic components is invariant under the Laplacian $L$. Some Jacobi sums naturally arise in the computation of the Smith normal form of $L$ restricted to these isotypic components. The description of $C_{p}$ is reduced to computation of $p$-adic valuations of Jacobi sums. Classical results by Stickelberger and Gauss describe the $p$-adic valuations of Jacobi sums in combinatorial terms. Theorem \ref{m} gives a description of $C_{p}$ in terms of $p$-adic valuations of Jacobi sums. At this point, writing the elementary divisor form of $C$ is now reduced to a counting problem. We use the transfer matrix method to determine the elementary divisor form in the case $\ell=3$.
 For a fixed $t$, Theorem \ref{e3} leads to a recursive algorithm that yields the $p$-elementary divisors of the critical group of $G(p, 3,t)$, where $p$ is any prime with $p \equiv 2\pmod{3}$. As a consequence we were able to show that the $p$-rank of the Laplacian of $G(p,3,t)$ is $\left(\dfrac{p+1}{3}\right)^{2t}(2^{t+1}-2)$ (see Cor. \ref{rank}). We were not able to obtain a similar results in the case $\ell \neq 3$.   

\section{Definitions and Notation.}\label{nota}
Let $(p,\ell)$ be a pair of primes such that $\ell >2$ and $p$ is primitive $\pmod{\ell}$. Let $t \in \bb{Z}_{>0}$ and $q=p^{(\ell-1)t}$. Moreover we assume that $\sqrt{q}=p^{(\ell-1)t/2} \neq \ell-1$ whenever $t$ is odd.
Consider the field $K=\bb{F}_{q}$ and the unique subgroup $S$ of $\K$ of order $k:= (q-1)/\ell$. Then by $G(p,\ell,t)$ we denote the graph with vertex set $K$ and edge set $ \{\{x,y\}\ |\ x,y \in K\ \text{and}\ x-y \in S\}$. This is the undirected Cayley graph associated with $(K,S)$. By $A$ we denote the adjacency matrix of $G(p,\ell,t)$ with respect to some fixed but arbitrary ordering of the vertex set $K$. The Laplacian matrix $L$ of $G(p, \ell,t)$ is the matrix $kI-A$. 

Given an Integral domain $\mathcal{R}$, by $\mathcal{R}^K$ we denote the $\mathcal{R}$-free module with $\{[x]| x \in K\}$ as a basis. Let $\mu_{A}$, $\mu_{L}$ be endmorphisms of $\mathcal{R}^K$ defined by 
$\mu_{A}([x]):=\sum\limits_{s \in S}[x+s]$ and $\mu_{L}([x]):=k[x]-\sum\limits_{s \in S}[x+s]$ respectively. The matrix representation of $\mu_{L}$ (respectively $\mu_{A}$) with respect to the basis set $\{[x]| x \in K\}$ is the Laplacian matrix $L$ (respectively $A$). 
The \emph{critical} group $C$ of $G(p,\ell ,t)$ is the finite part of the cokernal of $\mu_L:\bb{Z}^K \to \bb{Z}^K$. Let $C_{p}$ be the Sylow $p$-subgroup of $C$. Let $C_{p'}$ be the largest subgroup of $C$ whose order is not divisible by $p$. As $C$ is abelian, we have $C=C_{p}\oplus C_{p'}$.

\section{Some properties of $G(p,\ell,t)$.}\label{lin}  
In section 2 of \cite{VS}, the authors show that $G(p,\ell,t)$ is a strongly regular graph.
\begin{theorem*}[van Lint-Schrijver]
The graph $G(p,\ell,t)$ is a strongly regular graph with parameters
$$\left(q,\ \frac{q-1}{\ell}, \ \dfrac{q-3\ell+1 + (-1)^{t+1}(\ell-1)(\ell-2)\sqrt{q}}{\ell^{2}},\ \dfrac{q-\ell+1+(-1)^{t}(\ell-2)\sqrt{q}}{\ell ^{2}} \right),$$ where $q= p^{(\ell-1)t}$. The eigenvalues of the adjacency matrix $A$ of $G(p, \ell ,t)$ are $k=\dfrac{q-1}{\ell}$, $r_{\chi_{1}}$, $r_{\chi_{\alpha}}$, with multiplicities $1$, $k$, and $q-k-1$ respectively. Here $r_{\chi_{\alpha}}= \dfrac{-1+(-1)^{t}\sqrt{q}}{\ell}$ and $r_{\chi_{1}}=r_{\chi_{\alpha}}+(-1)^{t+1} \sqrt{q}$.  
\end{theorem*}

We now give a brief sketch of the proof of the above given in \S2 of \cite{VS}. We recall from \S\ref{nota} that $K=\bb{F}_{q}$ and $S$ is the unique subgroup of $\K$ of size $k$. Given $a \in S$, $b\in K$, $n \in \bb{Z}$, define $T_{(a,b,n)}:K \to K$ by $T_{(a,b,n)}(x):=ax^{p^{n}}+b$. Let $G$ be the group of transformations $T_{(a,b,n)}$. The action of $G$ on $K$ is shown to be a permutation action of rank $3$. It is also shown that the orbits of the natural action of $G$ on $K \times K$ are $\{(x,x)|\ x \in K\}$, $\Omega:=\{(x,y)|\ x,y \in K\ \text{and} \ x-y \in S\}$ and $\Delta:=\{(x,y)|\ x,y \in K\ \text{and}\ x-y \notin S\cup\{0\}\}$. The graph $G(p,\ell,t)$ is the graph with vertex set $K$ and edge set $\Omega$. Standard results on rank $3$ permutation groups of even order show that $G(p,\ell,t)$ is a strongly regular graph. 

Following notation in \S\ref{nota}, we have $\mu_{A}([x])=\sum\limits_{s \in S}[x+s]$. Let $\hat{K}$ be the group of complex valued characters of $K$. Given an additive character $\chi \in \hat{K}$, consider $[\chi]:=\sum\limits_{y \in K}\chi(y)[y]$ and $r_{\chi}=\sum\limits_{s \in S}\chi(s)$. We have $\mu_{A}([\chi])=r_{\chi}[\chi]$. By orthogonality of characters, we may observe that $\{[\chi]| \ \chi \in \hat{K}\}$ is a basis of $\bb{C}^K$ and thus that all eigenvalues of $\mu_{A}$ are of the form $r_{\chi}$ for some additive character $\chi$. Consider the character $\chi_{1}$ defined by $\chi_{1}(x):= e^{\frac{2\pi i Tr(x)}{p}}$. It is a well-known result that every character $\chi$ is of the form $\chi=\chi_{a}$, where $\chi_{a}(x)=\chi_{1}(ax)$, for $a\in K$. 
Let $\alpha$ be a generator of $\K$. We observe that for all $s \in S $, we have $r_{\chi_{s}}=r_{\chi_{1}}$; and for all $b \notin S \cup \{0\}$, we have $r_{\chi_{b}}=r_{\chi_{\alpha}}$. Thus the adjacency matrix has (at most) three eigenvalues $k=r_{\chi_{0}}$, $r_{\chi_{1}}$, $r_{\chi_{\alpha}}$, with geometric multiplicities $1$, $|S|=k$, and $q-k-1=|S\cup \{0\}|$ respectively. The parameters given in the Theorem above can now be deduced from the general theory of strongly regular graphs.

Now the eigenvalues of the Laplacian $L=kI-A$ are $0$, $u:=k-r_{\chi_{1}}$ and $v:=k-r_{\chi_{\alpha}}$, with multiplicities $0$, $k$, and $q-k-1$ respectively. We can see that $v= \sqrt{q} \dfrac{\sqrt{q}+(-1)^{t+1}}{\ell}$ and $u= v+(-1)^{t}\sqrt{q}$. It is well known that the nullity of the Laplacian matrix of a graph is equal to the number of connected components. Clearly $v \neq 0$, and thus $G(p,\ell,t)$ is connected if and only if either $t$ is even, or $t$ is odd and $\sqrt{q}\neq \ell-1$. We will assume throughout that $p^{(\ell-1)t/2} \neq \ell-1$ whenever $t$ is odd. 

Given an element $a$ in an unramified extension of of $\bb{Q}_{p}$, the $p$-adic valuation of $a$ is denoted by $v_{p}(a)$.
Let $v_{p}(\ell-1)=d$, then $v_{p}(u)=\frac{1}{2}(\ell-1)t+d$ and $v_{p}(v)=\frac{1}{2}(\ell-1)t$.

By Theorem $8.1.2$ of \cite{BH}, we have 
\begin{equation}\label{sreq}
L(L-(v+u)I)=vuI+\mu J,
\end{equation}
where $\mu=\dfrac{q-\ell+1+(-1)^{t}(\ell-2)\sqrt{q}}{\ell ^{2}}$.
Observing that $LJ=0$, we see that the minimal polynomial of $L$ is $(x)(x-u)(x-v)=0$. Therefore $L$ is diagonalizable. 
As a consequence of Kirchhoff's Matrix-Tree Theorem (cf. \cite{sta}), the order of critical group of $G(p,\ell,t)$ is $\dfrac{u^{k}v^{q-k-1}}{q}$.
\section{Main Results}\label{mainresult}
Let $(p,\ell)$ be a pair of primes with $\ell>2$ and $p$ primitive modulo $\ell$. Given $t \in \bb{N}$, let $q=p^{(\ell-1)t}$ and $k=\dfrac{q-1}{\ell}$. Let $C$ denote the critical group of $G(p ,\ell, t)$. Let $C_{p}$ be the Sylow $p$-subgroup of $C$. Let $C_{p'}$ be the largest subgroup of $C$ whose order is not divisible by $p$. As $C$ is abelian, we have $C=C_{p}\oplus C_{p'}$. 

 The following theorem describes the Sylow $p$-subgroup $C_{p}$ of the critical group $C$ of $G(p ,\ell, t)$.

\begin{theorem}\label{m}
Consider the graph $G(p,\ell,t)$ with $\sqrt{q}=p^{(\ell-1)t/2} \neq \ell-1$ whenever $t$ is odd. Let $d$ denote $v_{p}(\ell-1)$.
Given integers $a,b$ not divisible by $q-1$, let $c(a,b)$ denote the number of carries when adding the $p$-adic expansions of $a$ and $b$ $\pmod{q-1}$. Let $L$ be the Laplacian matrix and let $C$ be the critical group of $G(p,\ell,t)$.  
For $1\leq i \leq k-1$, let

 $$\mathfrak{min}(i)=\mathrm{min}\left(\{c(i+mk,nk)| 0 \leq m \leq \ell-1\ \text{and} \ 0 < n \leq \ell-1\}\right).$$ Given a non-zero positive integer $j$, let $e_{j}$ be the multiplicity of $p^{j}$ as a $p$-elementary divisor of $C$. By $e_{0}$ we denote the p-rank of the Laplacian $L$ of $G(p,\ell,t)$. Then the following are true.
\begin{enumerate}
\item $e_{0}=|\{i\ |\ 1\leq  i \leq k-1\ \text{and}\ \mathfrak{min}(i)=0 \}|+2$ and $e_{(\ell-1)t+d}=|\{i\ |\ \mathfrak{min}(i)=0 \}|$.

\item $e_{j}=|\{i\ | \  1\leq  i \leq k-1\ \text{and}\ \mathfrak{min}(i)=j \}|$ for $0<j <\frac{(\ell-1)t}{2}$.  
\item $e_{j}=e_{(\ell-1)t+d-j}$ for $0<j <\frac{(\ell-1)t}{2}$.
\item If $p\nmid \ell-1$, then $e_{\frac{(\ell-1)t}{2}}= q+1-2\sum\limits_{j<t} e_{j}$.
\item If $p\mid \ell-1$, then
\begin{enumerate}
\item $e_{\frac{(\ell-1)t}{2}+d}=k+2-\sum\limits_{j<t} e_{j}$ and
\item $e_{\frac{(\ell-1)t}{2}}= (\ell-1)k- \sum\limits_{j <t} e_{j}$.
 
\end{enumerate}   
\item  $e_{j}=0$ for all other $j$.
\end{enumerate}
\end{theorem}

We prove the above Theorem in \S\ref{syl}.

In the case of $G(p,3,t)$, application of the transfer matrix method (cf. Section 4.7 of \cite{EC1}) leads us to a recursive algorithm that outputs closed form expressions for multiplicities of $p$-elementary divisors of $C$. As a consequence, we also determine a closed form expression for the $p$-rank (i.e $e_{0}$ in the context of the Theorem above) of the Laplacian. The following theorem gives a quick recursive algorithm to compute $p$-elementary divisors. The proof of the following result is in \S\ref{3}.

Let
$P=\left( \left(\frac{p+1}{3}\right)^{2}(x^{2}y^{2}+x^{2}y+xy^{2}+x+y+1)+\left(\frac{p-2}{3}\right)^{2}3xy\right)$, $R=p^{2}x^{3}y^{3}$ and
 
 $Q=\left( \left(\frac{p+1}{3}\right)^{2}(xy)(x^{2}y^{2}+x^{2}y+xy^{2}+x+y+1)+\left(\frac{2p-1}{3}\right)^{2}3x^{2}y^{2}\right)$.  We define the polynomial $C(2t) \in \bb{C}[x,y]$ recursively as follows:

\begin{align}\label{rec}
\begin{split}
C(2)=2P \\ C(4)= 2(P^{2}-2Q),\\
 C(6)= 6R+2(P^{3}-2QP)-2PQ, \\
\text{and}\ C(2t)= PC(2t-2)-QC(2t-4) + R C(2t-6) &\ \text{for}\ t> 3.
\end{split}
\end{align}

\begin{theorem}\label{e3}
Let $C_{p}$ be the Sylow $p$-subgroup of the critical group of the graph $G(p,3,t)$ (with $(p,t) \neq (2,1)$). Given a non-zero positive integer $j$, let $e_{j}$ be the multiplicity of $ p^{j}$ as a $p$-elementary divisor of $C$. By $e_{0}$ we denote the p-rank of the Laplacian $L$ of $G(p,3,t)$. Let $e_{(a,b)}$ be the coefficient of $x^{a}y^{b}$ in  $C(2t)$.
 Then the following are true. (Here $\delta_{ij}$ is the Kronecker delta function.)
\begin{enumerate}
\item $e_{0}=e_{2t+\delta_{2,p}}+2= \left(\frac{(p+1)}{3}\right)^{2t}(2^{t+1}-2)$.
\item For $a<t$, we have $e_{a}=e_{2t+\delta_{2,p}-a}=\sum\limits_{a<b\leq t}e_{(a,b)}$
\item $e_{t+\delta_{2,p}}=(k+2-\sum\limits_{j<t} e_{j})+(1-\delta_{2,p})(2k- \sum\limits_{j <t} e_{j})$.
\item $e_{t}=(1-\delta_{2,p})(k+2-\sum\limits_{j<t} e_{j})+ (2k- \sum\limits_{j <t} e_{j}) $. 
\item $e_{a}=0$ for all other $a$. 
\end{enumerate}
\end{theorem}

Let $X$ be the complex character table of $K$ and $A$ the adjacency matrix of $G(p,\ell,t)$. Then all the entries of $X$ lie in $\bb{Z}[\zeta]$ for some primitive $p$th root of unity $\zeta$. We have by character orthogonality $\dfrac{1}{q}XX^{t}=I$ and 
\begin{equation}\label{orth}
\frac{1}{q} X A X^{t}= \mrm{diag}(r_{\psi})_{\psi},
\end{equation}
where $\psi$ runs over additive characters of $K$ and $r_{\psi}$ is as defined in \S\ref{lin}. Note that $r_{\psi}$ is an eigenvalue of $A$. We note that every prime $m \neq p$ is unramified in $\bb{Q}[\zeta]$. Let $\mathfrak{m}$ be a prime lying over $m$, then
the relation \eqref{orth} shows  similarity of matrices over the local PID $\left(\bb{Z}[\zeta]\right)_{\mathfrak{m}}$.
We can now conclude that $L=kI-A$ is similar to
$\mrm{diag}(0,\underbrace{u\ \ldots\ u}_{k \ \text{times}},\underbrace{v\ \ldots\ v}_{q-k-1 \ \text{times}})$, over $\left(\bb{Z}[\zeta]\right)_{\mathfrak{m}}$, for all primes $m\neq p$. This similarity implies $\left(\bb{Z}[\zeta]\right)_{\mathfrak{m}}$-equivalence of matrices. We have now proved the following result.

\begin{theorem}\label{pprime}
Consider the graph $G(p,\ell,t)$ with $p^{(\ell-1)t/2} \neq \ell-1$.  Let $C_{p'}$ be the largest subgroup of $C$ whose order is not divisible by $p$. Then $C_{p'}\cong \left(\frac{\bb{Z}}{u'\bb{Z}}\right)^{k} \times \left(\frac{\bb{Z}}{v'\bb{Z}}\right)^{q-k-1}$. Here $v'$ is the biggest divisor of $\sqrt{q} \dfrac{\sqrt{q}+(-1)^{t+1}}{\ell}$ that is coprime to $p$, and $u'$ is the biggest divisor of $u= v+(-1)^{t}\sqrt{q}$ that is coprime to $p$. 

\end{theorem} 

\begin{example}
Implementing the Recursion in \eqref{rec} in a computer algebra system such as Sage, we can compute $C(8)$. Now application of Theorems \ref{e3} and \ref{pprime} yield the critical groups of the family of graphs $(G(p,3,4))_{p}$, with $p$ running over primes primitive $\pmod{3}$.

The $2$-part of the critical group of $G(2,3,4)$ is $\prod\limits_{i=1}^{9}\left(\dfrac{\bb{Z}}{2^{i}\bb{Z}}\right)^{e_{i}}$, where
$[e_{i}]_{i=1}^{9}=[ 32, 8, 16, 84, 1, 16, 8, 32, 28]$.
The $2$-complement of the critical group of $G(2,3,4)$ is $\bb{Z}/ 15 \bb{Z}$. 

The Sylow $p$-subgroup of the critical group of $G(p,3,4)$ (with $p\neq 2$) $\prod\limits_{i=1}^{8}\left(\dfrac{\bb{Z}}{p^{i}\bb{Z}}\right)^{e_{i}(p)}$, where
\begin{enumerate}
\item $e_{8}(p)= 510\left(\dfrac{(p+1)}{3}\right)^{8}-2$,
\item $e_{1}(p)=e_{7}(p)=256/6561p^8 + 1040/6561p^7 + 1120/6561p^6 - 784/6561p^5 - 2240/6561p^4 -784/6561p^3 + 1120/6561p^2 + 1040/6561p + 256/6561$,
\item $e_{2}(p)=e_{6}(p)=776/6561p^8 +
592/6561p^7 - 2248/6561p^6 - 1904/6561p^5 + 320/6561p^4 -
1904/6561p^3 - 2248/6561p^2 + 592/6561p + 776/6561$,
\item $e_{3}(p)=e_{5}(p)=304/2187p^8 -
448/2187p^7 - 128/2187p^6 + 608/2187p^5 - 32/2187p^4 + 608/2187p^3
- 128/2187p^2 - 448/2187p + 304/2187$,
\item and $e_{4}(p)= 871/2187p^8 - 352/2187p^7 +
448/2187p^6 - 544/2187p^5 - 56/2187p^4 - 544/2187p^3 + 448/2187p^2
- 352/2187p + 871/2187$.
\end{enumerate}

The $p$-complement of the critical group of $G(p,3,4)$ (with $p\neq 2$) is $\bb{Z}/u'v' \bb{Z}$ ,where $u'= \dfrac{p^{4}-1}{3}$ and $v'=\dfrac{p^{4}+2}{3}$.  
\end{example}

\begin{remark}
For a fixed $t$, Theorem \ref{e3} implies that the multiplicities of the $p$-elementary divisors of the Laplacian  of $G(p,3,t)$ are polynomial expressions in $p$ of degree $2t$. We were however unable to extend the techniques in \S\ref{3} to prove similar results in the general case.       
\end{remark}
\section{Smith normal form}\label{Smith normal form}
Let $\mathfrak{R}$ be a Principal Ideal Domain, $\mathfrak{p} \in \mathfrak{R}$ a prime, and $Z:\mathfrak{R}^{m} \to \mathfrak{R}^{n}$ be a linear transformation. 
By the structure theorem for finitely generated modules over PIDs, we have $\{\alpha_{i}\}_{i=1}^{s} \subset \mathfrak{R} \setminus \{0\}$ such that $\alpha_{i} \mid \alpha_{i+1}$ and 
$$\mathrm{coker}(Z)\cong \mathfrak{R}^{n-s}\oplus \bigoplus\limits_{i=1}^{s} \mathfrak{R} /\alpha_{i}\mathfrak{R}.$$  
Let $[Z]$ denote the matrix representation of $Z$ with respect to the standard basis. Then the above equation tells us that we can find $P \in \mathrm{GL}_{n}(\mathfrak{R})$, and $Q \in \mathrm{GL}_{m}(\mathfrak{R})$ such that 
$$P[Z]Q=\left[
\begin{array}{c|c}
Y & 0_{(s \times n-s)} \\
\hline
0_{(m-s \times s)} & 0_{(n-s \times n-s)}
\end{array}
\right],$$ 
where $Y=\mrm{diag}(\alpha_{1} \ldots \alpha_{s})$. The diagonal form $P[Z]Q$ is called the Smith normal form of $Z$. Its uniqueness (up to multiplication of $\alpha_{i}$ by units) is also guaranteed by the aforementioned structure theorem. By invariant factors (elementary divisors) of $Z$, we mean the invariant factors (respectively elementary divisors) of the module $\mathrm{coker}(Z)$. 

The following is a well known result (for eg. see Theorem $2.4$ of \cite{sta}) that gives a description of the Smith normal form in terms of minor determinants.

\begin{lemma}\label{minor}
Let $Z$, $[Z]$, and $\{\alpha_{i}\}_{1\leq i \leq s}$ be as described above.
Given $1 \leq i \leq s$, let $d_{i}(Z)$ be the GCD of all $i\times i$ minor determinants of $[Z]$, and let $d_{0}(Z)=1$. We then have $\alpha_{i}=d_{i}([Z])/d_{i-1}([Z])$. 
\end{lemma}

Define $e_{j}(Z)=|\{\alpha_{i}| \ v_{\mathfrak{p}}(\alpha_{i})=j \}|$. Now $e_{j}(Z)$ is the multiplicity of $\mathfrak{p} ^{j}$ as $\mathfrak{p}$-elementary divisors of the $\mathfrak{R}$-module $\mathrm{coker}(Z)$. If $R=\bb{Z},$ then $e_{j}(Z)$ is the multiplicity of $ \mathfrak{p}^{j}$ as an elementary divisor of the abelian group $\mathrm{coker}(Z)$.

Let $\mathfrak{R}_{\mathfrak{p}}$ be the $\mathfrak{p}$-adic completion of $\mathfrak{R}$. We have 
$$\mathfrak{R}_{\mathfrak{p}}^{n}/T(\mathfrak{R}_{\mathfrak{p}}^{m})\cong \mathfrak{R}_{\mathfrak{p}}^{n-s}\oplus \bigoplus\limits_{j>0} \left(\mathfrak{R}_{\mathfrak{p}}/\mathfrak{p}^{j} \mathfrak{R}_{\mathfrak{p}}\right)^{e_{j}(\mathfrak{p})}.$$
Define $M_{j}(Z):=\{x \in \mathfrak{R}_{\mathfrak{p}}^{m}| \ Z(x) \in \mathfrak{p}^{j}\mathfrak{R}_{\mathfrak{p}}^{n}\}$. 
We have $\mathfrak{R}^{m}=M_{0}(Z)\supset M_{1}(Z) \supset \ldots \supset M_{n}(Z) \supset \cdots$. 

Let $\bb{F}=\mathfrak{R}_{\mathfrak{p}}/ \mathfrak{p} \mathfrak{R}_{\mathfrak{p}}$.
If $M \subset \mathfrak{R}_{\mathfrak{p}}^{m}$ is a submodule, define $\ov{M}=(M+\mathfrak{p} \mathfrak{R}_{\mathfrak{p}}^{m})/ \mathfrak{p} \mathfrak{R}_{\mathfrak{p}}^{m}$. Then  $\ov{M}$ is an $\bb{F}$-vector space. The following Lemma follows from the structure theorem.
\begin{lemma}\label{eldivcal}
$e_{j}(Z):=\dim(\ov{M_{j}(Z)}/\ov{M_{j+1}(Z)})$.
\end{lemma}  
So we have,
\begin{equation}\label{dim}
\dim(\ov{M_{j}(Z)})-\dim(\ov{\ker(Z)})=\sum\limits_{t\geq j} e_{t}(Z).  
\end{equation}
The following is Lemma $3.1$ of \cite{SD}. We include a short proof for the convenience of the reader.
\begin{lemma}\label{main}
Let $Z$, $\mathfrak{p}$, $M_{i}(Z)$, and $e_{i}(Z)$ be as defined above. Let $\kappa(Z)$ be the $\mathfrak{p}$-adic valuation of the product of a complete set of non-zero invariant factors of $Z$, counted with multiplicities. Suppose that we have two sequences of integers $0<t_{1}<t_{2} \ldots <t_{j}$ and $s_{1}>s_{2} \ldots >s_{j}>s_{j+1}=\dim(\ov{\ker(Z)})$ satisfying the following conditions.
\begin{enumerate}
\item $\dim(\ov{M_{t_{i}}(Z)})\geq s_{i}$ for $1\leq i \leq j$ 
\item $\kappa(Z)=\sum \limits_{i=1}^{j}(s_{i}-s_{i+1})t_{i}$,
\end{enumerate}
Then the following hold.
\begin{enumerate}[label=(\alph*)]
\item $e_{0}(Z)=m-s_{1}$.
\item $e_{t_{i}}(Z)=s_{i}-s_{i+1}$.
\item $e_{a}(Z)=0$ for $a \notin \{t_{1} \ldots t_{i}, \ldots t_{j}\}$.
\end{enumerate}
\end{lemma}
\begin{proof}
We have 
\begin{equation}\label{A}
\kappa(Z)=\sum\limits_{i\geq 1}ie_{i}(Z) 
  \geq \sum \limits_{k=1}^{j-1}\left(\sum\limits_{t_{k}\leq i < t_{k+1}} ie_{i}(Z)\right)+\sum_{i \geq t_{j}}ie_{i}(Z)
 \geq \sum \limits_{k=1}^{j-1}\left(t_{k}\sum\limits_{t_{k}\leq i < t_{k+1}} e_{i}(Z)\right)+t_{j}\sum_{i \geq t_{j}}e_{i}(Z).
\end{equation}

Application of equation \eqref{dim} given above yields 

\begin{align*} 
 \sum \limits_{k=1}^{j-1}\left(t_{k}\sum\limits_{t_{k}\leq i < t_{k+1}} e_{i}(Z)\right)+t_{j}\sum_{i \geq t_{j}}e_{i}(Z)
 & = \sum \limits_{k=1}^{j-1}\left(t_{k}(\dim(\ov{M_{t_{k}}(Z)})-\dim(\ov{M_{t_{k+1}}(Z)}))\right)+t_{j}\left(\dim(\ov{M_{t_{j}}(Z)})-\dim(\ov{\ker(Z)})\right).
 \end{align*}
 Now application of conditions (1) and (2) in the statement gives us
\begin{equation}\label{B}
 \sum \limits_{k=1}^{j-1}\left(t_{k}(\dim(\ov{M_{t_{k}}(Z)})-\dim(\ov{M_{t_{k+1}}(Z)}))\right)+t_{j}\left(\dim(\ov{M_{t_{j}}(Z)})-\dim(\ov{\ker(Z)})\right)\geq \sum \limits_{i=1}^{j}(s_{i}-s_{i+1})t_{i}=\kappa(Z).
\end{equation}
So the inequations \eqref{A} and \eqref{B} are in fact equations and thus the lemma follows.    
\end{proof}

The following result is $12.8.4$ of \cite{BH}.
\begin{lemma}\label{eigenval}
Let $Z: \mathfrak{R}^{n} \to \mathfrak{R}^{n}$ be a linear transformation and $\phi \in  \mathfrak{R}$ be an eigenvalue for $Z$, with geometric multiplicity $c$. Then $\dim(\ov{M_{v_{\mathfrak{p}}(\phi)}}(Z)) \geq c$. 
\end{lemma}
\begin{proof}
Let $Fr(R)$ be the field of fractions of $\mathfrak{R}_{\mathfrak{p}}$. We can extend $Z$ to a unique element of $\mrm{End}_{Fr(R)}(Fr(R)^{n})$. For convenience, let us denote this element by $Z$ as well.
Consider the eigenspace $V_{\phi}=\{x \in Fr(R)^{n}| Z(x)=\phi x\}$. Then $V_{\phi}\cap \mathfrak{R}_{\mathfrak{p}}^{n}$ is a pure $\mathfrak{R}_{\mathfrak{p}}$-submodule ($\mathfrak{R}_{\mathfrak{p}}$-direct summand) of
$\mathfrak{R}_{\mathfrak{p}}^{n}$ of rank $c=\dim(V_{\phi})$. It is clear that $V_{\phi}\cap \mathfrak{R}_{\mathfrak{p}}^{n} \subset M_{d}(Z)$. As $V_{\phi}\cap \mathfrak{R}_{\mathfrak{p}}^{n}$ is pure, we have $\ov{V_{\phi}\cap \mathfrak{R}_{\mathfrak{p}}^{n}} \subset \ov{M_{{v_{\mathfrak{p}}(\phi)}}(Z)}$. 
\end{proof}

\section{Character sums and block diagonal form of  $L$.}\label{block}
We recall the definition of the graph $G(p,\ell,t)$, its adjacency matrix $A$ and Laplaicain matrix $L$. This is the graph with vertex set $K=\bb{F}_{q}$ (with $q=p^{(\ell-1)t}$) and edge set $ \{\{x,y\}\ |\ x,y \in K\ \text{and}\ x-y \in S\}$, where $S$ is the subgroup of $\K$ with index $\ell$. By $C$, we denote the critical group of $G(p,\ell,t)$. We saw in \S\ref{lin} that $L$ has eigenvalues $0$, $v= \sqrt{q} \dfrac{\sqrt{q}+(-1)^{t+1}}{\ell}$ and $u= v+(-1)^{t}\sqrt{q}$, with multiplicities $1$, $q-k-1$ and $k$ respectively.  

Let $\xi$ be a primitive $(q-1)$-st root of unity in the algebraic closure of $\bb{Q}_{p}$. Then $\bb{Q}_{p}(\xi)$ is the unique unramified extension  of degree $(\ell-1)t$ over $\bb{Q}_{p}$. Let $R$ be the ring of integers in $\bb{Q}_{p}(\xi)$, then $pR$ is maximal in $R$ with $R/pR \cong \bb{F}_{q}=K$. 
Let $R^K$ denote the free $R$-module with $\{[x]|\ x\in K\}$ as a basis set. We recall from \S\ref{nota}, endomorphisms $\mu_{A}, \mu_{L}\in\mathrm{End}_{R}(R^K)$, with matrix representations (with respect to $\{[x]|\ x\in K\}$) $A$ and $L$ respectively. We have $\mu_{L}([x])=k[x]-\sum\limits_{s \in S}[x+s]$ and $\mu_{A}([x])= \sum\limits_{s \in S} [x+s],\ x \in K$. As $R$ is an unramified extension of $\bb{Z}_{p}$, the multiplicity of $p^{j}$ as an elementary divisor of $\mu_{L}$ is the same as that of $p^{j}$ as an elementary divisor of the integer matrix $L$.

Let $T:\K \to R^{\times}$ be the Teichm\"{u}ller character generating the cyclic group $\mrm{Hom}(\K,\ R^{\times})$. Then $\K$ acts $R^{K}$, which decomposes as the direct sum $R[0]\oplus R^{\K}$. Now the regular module $R^{\K}$ decompose further into a direct sum of $\K$-invariant submodules of rank $1$, affording the characters $T^{i}$, $i=0,\ldots,q-2$. The component affording $T^{i}$ is spanned by $f_{i}:= \sum\limits_{x \in \K} T^{i}(x^{-1})[x]$.     
Therefore $\{\mathbf{1}, f_{1} \ldots f_{q-2}, [0]\}$ is a basis for $R^K$, where $\mathbf{1}:=f_{0}+[0]=\sum\limits_{x \in K}[x]$.

Given an $R$-free $RS$-module $M$ and a character $\chi:S \to R^\times$, the isotypic component of $M$ corresponding to $\chi$ is the $RS$-submodule $M_{\chi}:=\{m \in M| \ sm= \chi(s)m\ \text{for all}\ s\in S \}$. 
For $0<j \leq k-1$, let $N_{j}$ denote the $R$-submodule of $R^K$ with $\{f_{j+mk}| \ 0\leq m \leq \ell-1\}$ as a basis set. Define $N_{0}$ to be the $R$-submodule of $R^{K}$ with $\{\mathbf{1}, [0], f_{k}, \ldots f_{(\ell-1)k}\}$ as a basis set. Then $N_{i}$ is the isotypic component for the character $T^{i}|_{S}$ of the group $S$. We now have 
\begin{equation}\label{de}
R^K=N_{0}\oplus N_{1}\ldots \oplus N_{k-1}.
\end{equation} 

Since $S$ is a group of automorphisms for $G(p, \ell, t)$, the $R$-linear maps $\mu_{A}$ and $\mu_{L}$ are in fact $RS$-module endomorphisms. It follows that $\mu_A$ and $\mu_L$ preserve the decomposition \eqref{de}. 
For $0 < i \leq k-1$, let $L_{i}$ denote the matrix of $\mu_{L}|_{N_{i}}$ with respect to the ordered basis $(f_{i+mk}| \ 0\leq m \leq \ell-1)$. Let $L_{0}$ be the matrix of $\mu_{L}|_{N_{0}}$ with respect to the ordered basis $(\mathbf{1}, [0], f_{k}, \ldots f_{(\ell-1)k})$. 
So with respect to the ordered basis $(\mathbf{1}, [0], f_{k}, \ldots f_{(\ell-1)k}) \cup_{i=1}^{k-1} (f_{i+mk}| \ 0\leq m \leq \ell-1)$, the matrix representation of the $R$-linear map $\mu_{L}$ is $\mathrm{diag}(L_{0},L_{1}, \ldots ,L_{k-1})$. We proved the following Lemma.

\begin{lemma}\label{blockd}
As $R$-matrices, $L$ is similar to the block diagonal matrix $\mathrm{diag}(L_{0},L_{1}, \ldots ,L_{k-1})$. 
\end{lemma}

Following conventions in \cite{Ax}, we extend the $T^{i}$'s to $K$. As per this convention, the character $T^{0}$ maps every element of $K$ to $1$, while $T^{q-1}$ maps $0$ to $0$. All other characters map $0$ to $0$. For two integers $a,b$ the Jacobi sum $J(T^{a},T^{b})$ is $\sum\limits_{x \in K}T^{a}(x)T^{b}(1-x)$. We refer the reader to Chapter $2$ of \cite{BKR} for formal properties of Jacobi sums. Following the conventions established, for $a \nequiv 0 \pmod{q-1}$, we have $J(T^{a},T^{0})=0$ and $J(T^{a},T^{q-1})=-1$.
 
The following Lemma describes action of $L_{i}$ on $N_{i}$.
\begin{lemma}\label{im} 
\begin{enumerate}
\item If $k \nmid i$, we have 
$\mu_{L}(f_{i})= \dfrac{1}{\ell}\left(qf_{i}-\sum\limits_{m=1}^{\ell-1} J(T^{-i},T^{-mk})f_{i+mk}\right)$.
\item For $1 \leq j \leq \ell-1$, we have
$\mu_{L}(f_{jk})=\dfrac{1}{\ell}\left(\mathbf{1} +qf_{jk} -\sum\limits_{m \neq -j,0} J(T^{-jk},T^{-mk})f_{jk+mk} - q[0]\right)$.
\item $\mu_{L}([0])=\dfrac{1}{\ell}\left(q[0]-\sum\limits_{m=1}^{\ell-1} f_{mk} -\mathbf{1} \right)$.
\item $\mu_{L}(\mathbf{1})=0$.
\end{enumerate}
\end{lemma} 
\begin{proof}
For $x \in K$, we have $\mu_{A}[x]=\sum\limits_{y \in S}[x+y]$.

Let $\delta_{S}$ denote the characteristic function of $S$, treated as a subset of $K$. We now have $\mu_{A}[x]= \sum\limits_{z \in K}\delta_{S}(z-x)[z]$. Writing  $\delta_{S}$ as a linear combination of characters of $S$, we have $\delta_{S} =\dfrac{1}{\ell} \left(\sum \limits_{m=0}^{\ell-1} T^{mk}-\delta_{0}\right)$. Here $\delta_{0}$ is $1$ at $0$ and $0$ elsewhere.

We have 
\begin{align}\label{1}
\ell \mu_{A}(f_{i})&= \ell \mu_{A}\left(\sum\limits_{x \in \K } T^{-i}(x)[x]\right) \notag \\
&= \ell \sum\limits_{x \in \K} T^{-i}(x) \sum\limits_{z \in K} \delta_{S} (z-x)[z] \notag\\
        &= \sum\limits_{x\in \K} T^{-i}(x) \sum\limits_{z \in K} T^{0}(z-x)[z] +\sum \limits_{m=1}^{\ell-1} \sum\limits_{x\in \K} T^{-i}(x) \sum\limits_{z \in K} T^{mk}(z-x)[z]- \sum\limits_{x\in \K} T^{-i}(x) \sum\limits_{z \in \K} \delta_{0}(z-x)[z]. 
\end{align}

From definition of $f_{i}$ and $\delta_{0}$, we have $\sum\limits_{x\in \K} T^{-i}(x) \sum\limits_{z \in \K} \delta_{0}(z-x)[z]=f_{i}$.

We recall from character theory that for a character $\chi$ of $\K$,
\begin{align}\label{s}
\sum_{x \in \K}\chi(x)&= \begin{cases}
q-1 \ \text{if $\chi$ is trivial, and}\\
0\ \text{otherwise.}
\end{cases}
\end{align}
Using \eqref{s}, we see that
\begin{equation}\label{2}
\sum\limits_{x\in \K} T^{-i}(x) \sum\limits_{z \in \K} T^{0}(z-x)[z]= \left(\sum\limits_{x \in \K}T^{-i}(x)\right)\left(\sum\limits_{z \in K}[z]\right)=0
\end{equation}

We now turn our attention to $\sum\limits_{x\in \K} T^{-i}(x) \sum\limits_{z \in K} T^{mk}(z-x)[z]$, with $1 \leq m \leq \ell-1$. We have 
\begin{equation} \label{3*}
\sum\limits_{x \in \K} T^{-i}(x) \sum\limits_{z \in K} T^{mk}(z-x)[z]  = \sum\limits_{(x,z)\in \K \times \K} T^{-i}(x)T^{mk}(z-x)[z]+ \sum\limits_{x\in \K} T^{-i}(x)T^{mk}(-x)[0]
\end{equation}

For $z \in \K$, we have $T^{-i}(x)T^{mk}(z-x)=T^{i}(z^{-1})T^{mk}(z)T^{-i}\left((x/z)\right) T^{mk}\left(1- (x/z)\right)$. Now, we have
\begin{align*}
\sum\limits_{(x,z)\in \K \times \K} T^{-i}(x)T^{mk}(z-x)[z] &= \sum\limits_{(x,z)\in \K \times \K} T^{i}(z^{-1})T^{mk}(z)T^{-i}\left((x/z)\right) T^{mk}\left(1- (x/z)\right)[z]\\
&= \sum\limits_{(y,z)\in \K \times \K} T^{i-mk}(z^{-1})T^{-i}(y)T^{mk}(1- y)[z]\\
&= \left(\sum\limits_{y \in \K}T^{-i}(y)T^{mk}(1- y)\right) \times \left(\sum\limits_{z\in \K}T^{i-mk}(z^{-1})[z] \right)\\
&= J(T^{-i},T^{mk})f_{i-mk}.
\end{align*}
Using the above equality along with \eqref{1}, \eqref{2}, \eqref{3*}, we have
\begin{equation}\label{last1}
\ell \mu_{A}(f_{i})= \sum\limits_{m=0}^{\ell-1} J(T^{-i},T^{mk})f_{i-mk} + \sum\limits_{m=1}^{\ell-1}\sum\limits_{x\in \K} T^{-i}(x)T^{mk}(-x)[0] -f_{i}
\end{equation}
As $-1 \in S$ and $[\K:S]=k$, we have $T^{mk}(-x)=T^{mk}(x)$. Thus the middle sum above is 
$\left(\sum\limits_{m=1}^{\ell-1}\sum\limits_{x \in \K}T^{i-mk}(x)\right)[0]$. Using this along with \eqref{s} in \eqref{last1}, we conclude that
\begin{enumerate}[label=\alph*)]
\item if $k \nmid i$, we have $\ell \mu_{A}(f_{i}) =\sum\limits_{m=1}^{\ell-1} J(T^{-i},T^{-mk})f_{i+mk} -f_{i} $, and; 
\item for $1 \leq j \leq \ell-1$, we have $\ell \mu_{A}(f_{jk})=\sum\limits_{m=1}^{\ell-1} J(T^{-jk},T^{mk})f_{jk-mk} +(q-1)[0] -f_{jk}$.
\end{enumerate}
Using $L=kI-A$ now readily yields (1). From the general theory of Jacobi sums, we have for any character $\lambda$, $J(\lambda, \lambda ^{-1})= -\lambda(-1)$. Since $-1 \in S$, we have $T^{jk}(-1)=1$, and therefore we have $J(T^{-mk},T^{mk})=-1$. Thus $\ell \mu_{A}(f_{jk})=\sum\limits_{m=1}^{\ell-1} J(T^{-jk},T^{mk})f_{jk-mk} +(q-1)[0] -f_{jk}= \sum\limits_{m \neq -j,0}J(T^{-jk},T^{mk})f_{jk-mk} -\mathbf{1}+q[0]$. Now the (2) follows by using $L=KI-A$. 

The proof of the remaining statements is straight forward. 
\end{proof}

We observed in \S\ref{lin} that $L$ is diagonalizable and thus so are $L_{i}$'s.
We recall from \S3 that the eigenvalues of $L$ are  $0$, $u$ and $v$, with multiplicities $1$, $k$ and $(\ell-1)k$ (same as $q-k-1$), respectively. Again from \S\ref{lin} we know that the nullity of $L$ is $1$.
Now since the nullity of $L_{0}$ is $1$ (c.f Lemma \ref{im}), all other $L_{i}$'s are invertible. It follows that for $i \neq 0$, the characteristic polynomial of $L_{i}$ is a polynomial of the form $(x-u)^{a}(x-v)^{b}$ with $a+b=\ell$. By Lemma \ref{im} and diagonalizability of $L_{i}$, we have $q=tr(L_{i})=au+bv$. It now follows that $a=1$ and $b=\ell-1$. By similar arguments, we may show that the eigenvalues of $L_{0}$ are $0$, $u$ and $v$ with geometric multiplicities $1$, $1$ and $\ell-1$, respectively. We have proved the following Lemma.
\begin{lemma}\label{eig}
\begin{enumerate}
\item For $i\neq 0$, the eigenvalues of $L_{i}$ are $u$ and $v$ with geometric multiplicities $1$ and $\ell-1$, respectively.
\item The eigenvalues of $L_{0}$ are $0$, $u$ and $v$ with geometric multiplicities $1$, $1$ and $\ell-1$, respectively.
\end{enumerate}
\end{lemma}  
\section{The Sylow $p$-subgroup of the critical group of $G(p,\ell,t)$}\label{syl}
By Lemma \ref{blockd}, it is clear that finding the elementary divisors of $R$-matrices $L_{i}$'s will determine the $p$-elementary divisors of the critical group. 
As $\ell$ is a unit in $R$, the Smith normal form of $L_{i}$ is the same as that of $\ell L_{i}$. Lemma \ref{im} shows that the any entry of $\ell L_{i}$ is either $q$ or is a Jacobi sum of the form $J(T^{-(i+mk)},T^{-nk})$, where $0\leq m \leq \ell-1$ and $0 < n \leq \ell-1$. In the context of Lemma \ref{minor}, it is worth investigating the $p$-adic valuations of Jacobi sums.  

An integer $a$ not divisible by $q-1$ has, when reduced modulo $q-1$, a unique $p$-digit expansion $a \equiv a_{0}+a_{1}p +\ldots+a_{(\ell-1)t-1}p^{(\ell-1)t-1} \pmod{q-1}$, where $0 \leq a_{i}\leq p-1$. We represent this expansion by the tuple of digits $(a_{0},\ldots, a_{i}, \ldots,a_{(\ell-1)t-1})$.
By $s(a)$ we denote the sum $\sum a_{i}$. For example, $1$ has the expansion $(1, \ldots, 0 ,\ldots 0)$ and $s(1)=1$.

Applying Stickelberger's theorem on Gauss Sums \cite{Stick} and the well know relation between Gauss and Jacobi sums we can deduce the following theorem.
\begin{theorem}\label{stick}
Let $q$ be a power of a prime $p$ and let $a$ and $b$ be integers not divisible by $q-1$. If $a+b \nequiv 0 \pmod{q-1}$, then we have 
$$v_{p}(J(T^{-a},T^{-b}))= \frac{s(a)+s(b)-s(a+b)}{p-1}.$$  
In other words, the $p$-adic valuation of $J(T^{-a},T^{-b})$ is equal to the number of carries, when adding $p$-expansions of $a$ and $b$ modulo $q-1$.
\end{theorem}

Given $b \in \bb{Z}$, by $[b]$ denote the unique positive integer less than $\ell$ satisfying $b \equiv [b] \pmod{\ell}$. We can now see that
\begin{align*}
k &=\frac{q-1}{\ell}\\
&= \frac{p^{\ell-1}-1}{\ell} \times \frac{p^{(\ell-1)t}-1}{p^{\ell-1}-1}\\
&= \sum\limits_{i=0}^{\ell-2} \left(\frac{[p^{i}]p-[p^{i+1}]}{\ell}\right)p^{\ell-2-i} \times \sum\limits_{i=0}^{t-1} p^{(\ell-1)i}.
\end{align*}
Thus in the notation we adopted, the tuple for $k$ is the tuple in which the string $$\left(\dfrac{[p^{\ell-2}]p-1}{\ell}, \ldots, \dfrac{[p^{i}]p-[p^{i+1}]}{\ell}, \ldots ,\dfrac{p-[p]}{\ell}\right)$$ repeats $t$ times. As $p$ is primitive modulo $\ell$, we have $\{[p^{i}]| 0 \leq i \leq \ell-2\}=\{1,2, \ldots \ell-1\}$. We can now conclude that $s(k)= t\sum\limits_{i=0}^{\ell-2} \left(\frac{[p^{i}]p-[p^{i+1}]}{\ell}\right)=\frac{(\ell-1)t}{2}(p-1)$. 

Now for $0 \leq i,j \leq \ell-1$, we can find $r_{i,j} \in \bb{Z}$ such that $[p^{i}][p^{j}]=[p^{i+j}]+r_{i,j}\ell$. Given any $m \in \{1,2, \ldots, \ell-1\}$, as $p$ is primitive $\pmod{\ell}$, there is a positive integer $j$ such that $[p^{j}]=m$. We have
$mk=\left([p^{j}]\frac{p^{\ell-1}-1}{\ell}\right) \times \frac{p^{(\ell-1)t}-1}{p^{\ell-1}-1}$. Now we have
\begin{align*}
\left([p^{j}]\frac{p^{\ell-1}-1}{\ell}\right) &= \sum\limits_{i=0}^{\ell-2} [p^{j}]\left(\frac{[p^{i}]p-[p^{i+1}]}{\ell}\right)p^{\ell-2-i}\\
&= \sum\limits_{i=0}^{\ell-2} \left(\frac{[p^{i+j}]p+pr_{i,j}\ell-[p^{i+1+j}]-r_{i+1,j}\ell}{\ell}\right)p^{\ell-2-i}\\
&= \sum\limits_{i=0}^{\ell-2} \left(\frac{[p^{i+j}]p-[p^{i+1+j}]}{\ell}\right)p^{\ell-2-i}+\sum \limits_{i=0}^{\ell-2} r_{i,j}p^{\ell-1-i} - \sum\limits_{i=0}^{\ell-2} r_{i+1,j}p^{\ell-2-i}\\
&= \sum\limits_{i=0}^{\ell-2} \left(\frac{[p^{i+j}]p-[p^{i+1+j}]}{\ell}\right)p^{\ell-2-i} +r_{0,j}p^{\ell-1}-r_{\ell-1,j}.
\end{align*} 
As $r_{0,j}=r_{\ell-1,j}=0$, from the above computation we observe that
\begin{equation*}
\left([p^{j}]\frac{p^{\ell-1}-1}{\ell}\right) \times \frac{p^{(\ell-1)t}-1}{p^{\ell}-1} = 
\sum\limits_{i=0}^{\ell-2} \left(\frac{[p^{i+j}]p-[p^{i+1+j}]}{\ell}\right)p^{\ell-2-i} \times \sum\limits_{i=0}^{t-1} p^{(\ell-1)i} 
\end{equation*}
Thus in the notation we adopted, the tuple for $[p^{j}]k$ is the tuple in which the string $$\left(\dfrac{[p^{\ell+j-2}]p-1}{\ell}, \ldots, \dfrac{[p^{i+j}]p-[p^{i+j+1}]}{\ell}, \ldots ,\dfrac{[p^{j}]p-[p^{j+1}]}{\ell}\right)$$ repeats $t$ times. So the digits of $mk=[p^{j}]k$ can be obtained by permuting the digits of $k$, and thus $s(mk)=s(k)=\frac{(\ell-1)t}{2}(p-1)$.

Given $a,b$ as described in the theorem above, by $c(a,b)$ we denote $v_{p}(J(T^{-a},T^{-b}))$. Then by Lemma \ref{im} the off-diagonal entries of $L_{i}$ (with $i \neq 0$) are $u_{mn}p^{c(i+mk,nk)}$ for some units $u_{mn}$ of $R$, and the diagonal entries are all $q/\ell$. 
Lemma \ref{eig} shows that $L_{i}$ satisfies $(x-u)(x-v)=0$. We make use of this to arrive at the following lemma.

\begin{lemma}\label{palin}
Given $j < \frac{(\ell-1)t}{2}$ and $0<i\leq k-1$, the multiplicity of $p^{j}$ as an elementary divisor of $L_{i}$ is the same as that of $p^{v_{p}(uv)-j}$. 
\end{lemma}
\begin{proof}
As $L_{i}$ satisfies $(x-u)(x-v)=0$, we have $(L_{i})(L_{i}-(v+u)I)=vuI$. Let $P$ and $Q$ be $R$-matrices such that $PL_{i}Q$ is the Smith normal form of $L_{i}$. Now consider 
$PL_{i}QQ^{-1}(L_{i}-(v+u)I)P^{-1}=vuI$. This shows that the multiplicity of $p^{v_{p}(uv)-j}$ as an elementary divisor of $L_{i}$ is the same as the multiplicity of $p^{j}$ as an elementary divisor of $L_{i}-(v+u)I$. Since $L_{i}$ and $L_{i}-(u+v)I$ are congruent modulo $p^{v_{p}(v)}=p^{(\ell-1)t/2}$, for $0 \leq j <(\ell-1)t/2$, the multiplicity of $p^{j}$ as an elementary divisor of $L_{i}-(v+u)I$ is the same as the multiplicity of $p^{j}$ as an elementary divisor of $L_{i}$.   
\end{proof}
We now compute the Smith normal forms of $L_{i}$'s.
\begin{lemma}\label{snfl}
\begin{enumerate}[label=(\roman*)]
\item  For $0<i \leq \ell-1$, the Smith normal form of $L_{i}$ over $R$ is the diagonal matrix

 $\mrm{diag}(p^{\mathfrak{min}(i)},\underbrace{p^{(\ell-1)t/2},\ldots,p^{(\ell-1)t/2}}_{(\ell-2) \ \text{repetitions}},p^{v_{p}(uv)-\mathfrak{min}(i)})$.
 
 Here $\mathfrak{min}(i)=\mathrm{min}\left(\{c(i+mk,nk)| 0 \leq m \leq \ell-1\ \text{and} \ 0 < n \leq \ell-1\}\right)$.
\item The Smith normal form of $L_{0}$ over $R$ is 
$\mrm{diag}(1,1,\underbrace{p^{(\ell-1)t/2}}_{\ell-3\ \text{times}},p^{v_{p}(u)},0)$. 
\end{enumerate}
\end{lemma}
\begin{proof}
Given an $R$-matrix $X$, by $\kappa(X)$, we denote the $p$-adic valuation of the product of a complete set of non-zero invariant factors of $X$, counted with multiplicities. By the notation in \S\ref{Smith normal form}, $e_{i}(X)$ denotes the multiplicity of $p^{i}$ as an elementary divisor of $X$. Following the notation in \S\ref{Smith normal form}, we consider the vector spaces $\ov{M_{y}(X)}$ with $y \in \bb{Z}_{\geq 0}$. We use Lemma \ref{main} to prove our results.

As a consequence of Kirchhoff's Matrix-Tree theorem, we have $\kappa(L)=v_{p}(|C|)=v_{p}\left(\frac{u^{k}v^{q-k-1}}{q}\right)=kv_{p}(u)+(\ell-1)kv_{p}(v)-v_{p}(q)$.
Lemma \ref{eig} implies that for $i\neq 0$, we have $\kappa(L_{i})=v_{p}(\mrm{det}(L_{i}))=v_{p}(u)+(\ell-1)v_{p}(v)$. Application of Lemma \ref{blockd} gives us $\kappa(L_{0})=\kappa(L)-\sum_{i\neq 0} \kappa(L_{i})=v_{p}(u)+(\ell-3)v_{p}(v)$.

(i) By Theorem \ref{stick}, we have $c(i+mk,nk)+c(i+(m+n)k, (\ell-n)k)= (\ell-1)t$. We can now conclude that $\mathfrak{min}(i) \leq (\ell-1)t/2$. Let $\mrm{diag}(\beta_{1}, \beta_{2}, \ldots \beta_{\ell})$ be the Smith normal form of $L_{i}$.
Then by Lemma \ref{im} and Lemma \ref{minor}, it follows that $\mathfrak{min}(i)=v_{p}(\beta_{1})$. By definition of $M_{\mathfrak{min}(i)}(L_{i})$ it follows that $M_{\mathfrak{min}(i)}(L_{i})=N_{i}$ and thus $\dim(\ov{M_{\mathfrak{min}(i)}(L_{i})})=\ell$. 
First we assume that $\mathfrak{min}(i)<(\ell-1)t/2$. In this case, by Lemma \ref{palin} we have $e_{v_{p}(uv)-\mathfrak{min}(i)}(L_{i})=e_{\mathfrak{min}(i)}(L_{i}) \geq 1$ and thus $\dim(\ov{M_{v_{p}(uv)-\mathfrak{min}(i)}(L_i)}) \geq 1$. 
Lemma \ref{eig} tell us that geometric multiplicity of $v$ as an eigenvalue of $L_{i}$ is $\ell-1$. Now Lemma \ref{eigenval} implies that $\dim(\ov{M_{(\ell-1)t/2}(L_{i})}) \geq \ell-1$. Applying \eqref{dim}, we have $\dim(\ov{M_{\mathfrak{min}(i)}(L_{i})})-\dim(\ov{M_{(\ell-1)t/2}(L_i)})=\sum\limits_{\mathfrak{min}(i)\leq j < (\ell-1)t/2}e_{j}(L_{i}) \geq 1$, and thus $\dim(\ov{M_{\mathfrak{min}(i)}(L_{i})})\geq \ell$. 
Therefore by Lemma \ref{main}, setting $j=3$, $s_{1}=\ell, s_{2}=\ell-1, s_{3}=1, s_{4}=\ov{\ker(L_{i})}=0$, $t_{1}=\mathfrak{min}(i), t_{2}=(\ell-1)t/2$, and $t_{3}=v_{p}(uv)-\mathfrak{min}(i)$, we have $e_{\mathfrak{min}(i)}(L_{i})=e_{v_{p}(uv)-\mathfrak{min}(i)}(L_{i})=1$, $e_{(\ell-1)t/2}(L_{i})=\ell-2$, and $e_{i}(L_{i})=0$ for all other $i$.
Now assume that $\mathfrak{min}(i)=(\ell-1)t/2$. Lemma \ref{eigenval} implies that $\dim(\ov{M_{v_{p}(u)}(L_{i})}) \geq 1$, since $u$ is an eigenvalue of multiplicity $1$. Therefore by Lemma \ref{main}, setting $j=2$, $s_{1}=\ell$, $s_{2}=1$, $s_{3}=\ov{\ker(L_{i})}$, $t_{1}=\mathfrak{min}(i)=(\ell-1)t/2$, and $t_{2}=v_{p}(u)= v_{p}(uv)-(\ell-1)t/2$, we have $e_{(\ell-1)t/2}(L_{i})=\ell-1$, $e_{v_{p}(uv)-\mathfrak{min}(i)}(L_{i})=1$, and $e_{i}(L_{i})=0$ for all other $i$. Thus we have (i).

(ii) By Lemma \ref{im} and Theorem \ref{stick}, there are units $v_{(mn)}$ in $R$ such that the matrix $\ell L_{0}$ is 
$$
\left[
\begin{array}{c c c c |c c}
q & v_{(12)}\sqrt{q} & \ldots & v_{(1\ \ell-1)}\sqrt{q} & -1 & 0 \\
\vdots & \ddots & \ldots & \vdots &\vdots & \vdots\\
v_{(\ell-1\ 1)}\sqrt{q}& \ldots & \ldots &q & -1 & 0 \\
\hline
-q & \ldots & \ldots & -q & q & 0 \\
1 & \ldots & \ldots & 1 & -1 & 0

\end{array}
\right].
$$ 

The determinant of the $2 \times 2$ minor $\left[ \begin{array}{c c}
q & -1 \\
1 & -1
\end{array}  \right]$ of $\ell L_{0}$ is a unit in $R$. Observe that any $3 \times 3$ minor of $\ell L_{0}$ has $p$-valuation of atleast $v_{p}(q)$. Now applying Lemma \ref{minor} yields that the multiplicity of $p^{0}=1$ as an elementary divisor of $L_{0}$ is $2$, that is $e_{0}(L_{0})=2$. Now Lemma \ref{eldivcal} implies that $\dim(\ov{M_{0}(L_{0})})-\dim(\ov{M_{1}(L_{0})})=2$, and thus we have $\dim(\ov{M_{1}(L_{0})})=\ell+1-2=\ell-1$. By Lemma \ref{eig} and Lemma \ref{eigenval}, we have $\dim(\ov{M_{v_{p}(v)}(L_{0})}) \geq \ell-1$. Since $\ov{M_{1}(L_{0})} \supset \ov{M_{v_{p}(v)}(L_{0})}$, we have $\dim(\ov{M_{v_{p}(v)}(L_{0})})=\ell-1$. Lemma \ref{im} implies that $\ov{\Ima(L_{0})}$ is generated by $\mathbf{1}$ and $\sum\limits_{j \neq 0} f_{jk}+\mathbf{1}$. Therefore $\dim(\ov{\Ima(L_{0})})=2$. As $LJ=0$, by \eqref{sreq} the restriction of $L$ to $\Ima(L)$ satisfies $L(L-v+uI)=vuI$. As $\Ima(L_{0}) \subset Im(L)$, we can conclude that $\ov{\Ima(L_{0})} \subset \ov{M_{v_{p}(uv)}(L_{0})} \subset \ov{M_{v_{p}(u)}(L_{0})}$. 

We have $\kappa(L_{0})=v_{p}(v)(\ell-1-2)+v_{p}(u)(2-1)=\frac{(\ell-1)t}{2}(\ell-1-2)+v_{p}(u)(2-1)$. Now application of Lemma \ref{main} yields (ii). 
\end{proof}

\subsection{Proof of Theorem \ref{m}}
Lemma \ref{blockd} shows the Laplacian matrix $L$ is similar over $R$ to the block diagonal matrix $\mrm{diag}(L_{0}, L_{1}, L_{2} \ldots L_{k-1})$. Results $(1)$, $(2)$, $(3)$, and $(6)$ now follow by applying Lemma \ref{snfl}. 
 
If $p \nmid \ell-1$, we have $d=v_{p}(\ell-1)=0$ and $v_{p}(u)= \frac{(\ell-1)t}{2}=v_{p}(v)=v_{p}(uv)-\frac{(\ell-1)t}{2}$. From \eqref{dim}, we have $q-1=\sum\limits_{j \neq \frac{(\ell-1)t}{2}}e_{j} +e_{\frac{(\ell-1)t}{2}}$. Now application of $(1)$ and $(3)$ yields $(4)$.

If $p\mid \ell-1$, then $v_{p}(u)>\frac{(\ell-1)t}{2}=v_{p}(v)$ and $v_{p}(u)=v_{p}(uv)-\frac{(\ell-1)t}{2}$. Now by Lemma \ref{blockd} and Lemma \ref{snfl}, we deduce that 
$e_{v_{p}(u)}=|\{i|\ 1\leq i\leq k-1\ \text{and}\ \mathfrak{min}(i)=\frac{(\ell-1)t}{2}\}|+1$. We may also deduce that $e_{\frac{(\ell-1)t}{2}}=(\ell-2)|\{i|\ 1\leq i\leq k-1\ \text{and}\ \mathfrak{min}(i)=\frac{(\ell-1)t}{2}\}|$ + $(\ell-3)|\{i|\ 1\leq i\leq k-1\ \text{and}\ c(i)<\frac{(\ell-1)t}{2}\}|$+ $\ell-3$ and thus that 
$e_{v_{p}(u)}-1+ (\ell-2)(k-1)+ (\ell-3)=e_{\frac{(\ell-1)t}{2}}$. From \eqref{dim}, we have $q-1= e_{v_{p}(u)}+e_{\frac{(\ell-1)t}{2}}+ \sum\limits_{j\notin\{\frac{(\ell-1)t}{2},v_{p}(u)\}}e_{j}$. Now application of $(1)$ and $(3)$ yield $(5)$. 

\section{The Critical group of $G(p,3,t)$}\label{3}
We now turn our focus to graphs of the form $G(p,3,t)$. We assume that $(p,t) \neq (2,1)$ and $p \equiv 2 \pmod{3}$, so these graphs are connected and strongly regular. 
Recall that this is the Cayley graph on the additive group of the field $K=\bb{F}_{q}$ ($q=p^{2t}$) with ``connection'' set $S$, where $S$ is the unique subgroup of $\K$ satisfying $k:=|S|=\frac{q-1}{3}$. 
All the results in the previous sections transfer to this case by setting $\ell=3$.

Lemma \ref{snfl} shows that for $i \neq 0$, the Smith normal form of $L_{i}$ over $R$ is $\mrm{diag}(p^{c},p^{t},p^{v_{p}(uv)-c})$. Here $c$ is the least among the $p$-adic valuations of the entries of $L_{i}$. In this case,  $v= \sqrt{q} \dfrac{\sqrt{q}+(-1)^{t+1}}{3}$ and $u= v+(-1)^{t}\sqrt{q}$ are the non-zero eigenvalues of the Laplacian of $G(p,3,t)$.

Given integers $a,b$ not divisible by $q-1$, let $c(a,b)$ denote the number of carries when adding the $p$-adic expansions of $a$ and $b$ $\pmod{q-1}$. Consider the following counting problem.

\paragraph{Counting Problem:} For $1\leq i \leq k-1$, by $\mathfrak{min}(i)$ we denote $\mrm{min}\left(\{c(i+mk,nk)| \ 0\leq m \leq 2, \ \text{and} \ n=1,2\}\right)$. Given $0 \leq a < t$, find $|\{i| \ \mathfrak{min}(i)=a\}|$.

Given a positive integer $a$, by $e_{a}$ we denote the multiplicity of $p^{a}$ as an elementary divisor of the critical group of $G(p,3,t)$. Let $e_{0}$ be the $p$-rank of the Laplacian of $G(p,3,t)$.  
Theorem \ref{m} implies that, for $0<a<t$, we have $e_{a}= |\{i| \ \mathfrak{min}(i)=a\}|$, and $e_{0}=|\{i| \ \mathfrak{min}(i)=0\}|+2$. 
Thus the solution to this problem will immediately provide us with the $p$-elementary divisors of the critical groups of graphs of the form $G(p,3,t)$. 

Every integer $a$ that is not divisible by $q-1$, when reduced modulo $q-1$, has a unique $p$-adic expansion $a\equiv \sum\limits_{m=0}^{2t-1} a_{m}p^{m} \pmod{q-1}$, where $0\leq a_{m} \leq p-1$. By $s(a)$, we denote $\sum a_{m}$. The $p$-adic expression for $k$ is $\sum\limits_{m=0}^{t-1}\left(\frac{p-2}{3}\right)p^{2m}+ \left(\frac{2p-1}{3}\right) p^{2m+1}$ and that of $2k$ is $\sum\limits_{m=0}^{t-1}\left(\frac{2p-1}{3}\right) p^{2m}+ \left(\frac{p-2}{3}\right)p^{2m+1} $. Thus we have $s(k)=s(2k)=t(p-1)$.

 We may observe from Theorem \ref{stick} that 
\begin{equation}\label{for}
c(a,b)= \dfrac{s(a)+s(b)-s(a+b)}{p-1}.
\end{equation}

 Given $j \in  \bb{Z}$, by $\ov{j}$ we denote the unique element of $\{0, 1, \ldots, q-2\}$ satisfying $j \equiv \ov{j} \pmod{q-1}$.

The following follows from \eqref{for}.

\begin{lemma}\label{comp}
Given an integer $j\not\equiv 0 \pmod{q-1}$ and $m=-1,1$, the following hold.   
\begin{enumerate}
\item $c(\ov{j},\ov{mk})+c(\ov{j+mk},\ov{-mk})=2t$
\item $c(\ov{j},\ov{mk})+c(\ov{j+mk},\ov{mk})= t+ c(\ov{j},\ov{-mk})$
\item $c(\ov{j},\ov{mk})=c(\ov{-j-mk},\ov{mk})$
\end{enumerate}
\end{lemma}  
Let $j \in \{1,\ldots,q-2\}\setminus \{k,2k\}$, define $g(j):=\{c(j,k),c(j,2k)\}$. For every $j$, there is a unique $\phi(j) \in \{1,2 \ldots,k-1\}$ such that $j-\phi(j) \in \{0,k,2k\}$. Note that $\phi^{-1}(i) = \{i,i+k,i+2k\}$.

For $0\leq a\leq t$, we define $Y_{a}:=\{j \ \lvert\ g(j)=\{a,b\}\ \text{for some $b$ such that $a \leq b \leq t$} \}$ and $R_{a}=\{i| 1\leq i \leq k-1\ \text{and}\ \mathfrak{min}(i)=a\}$. From Theorem \ref{m}, we have i) $e_{a}=|R_{a}|$, for $0< a <t$ and; ii) $e_{0}=|R_{0}|+2$. 

\begin{lemma}\label{trs}
Given $Y_{a}$ and $\phi$ defined above and $a<t$, the following are true.
\begin{enumerate}
\item If $\phi_{a}$ is the restriction of $\phi$ to $Y_{a}$, then $\phi_{a}(Y_{a})=R_{a}$. 
\item Let $i \in R_{a}$. If $j \in \phi_{a}^{-1}(i)$ and $m_{j} \in \{1,\ 2\}$ such that $c(j,m_{j}k)=a$, then
\begin{enumerate}
\item $\ov{j+m_{j}k} \notin \phi_{a}^{-1}(i)$
\item $\phi_{a}^{-1}(i)=\{j\}$ if and only if $t \notin g(j)$;
\item and $\phi_{a}^{-1}(i)=\{j,\ov{j-m_{j}k}\}$ if and only if $t\in g(j)$.  
\end{enumerate}
\item For $0\leq a<t$, we have $|R_{a}|=|Y_{a}|-\frac{1}{2}|\{j|\ g(j)=\{a,t\}\}|=|Y_{a}|-|\{j|\ g(j)=\{a\}\}|$

 \end{enumerate}

\end{lemma}
\begin{proof}
1) Let $m \in \{1,2\}$ and $j \in Y_{a}$ such that $c(j,mk)=a$ and $c(j,\ov{-mk})=b$. Then by Lemma \ref{comp}, we have
$\{c(\ov{j+mk},nk)| \ 0\leq m \leq 2, \ \text{and} \ n=1,2\}=\{a,b,t-a+b,t-b-a, 2t-a,2t-b\}$. Since $a\leq b \leq t$, we have $\mathfrak{min}(\phi(j))=a$ and thus $\phi_{a}(Y_{a}) \subset R_{a}$. If $i \in R_{a}$, then there exists $j \in \{i,i+k,i+2k\}$ and $m \in {1,2}$ such that $c(j,mk)=a$. Using $a= \mrm{min}\left(\{c(i+mk,nk)| \ 0\leq m \leq 2, \ \text{and} \ n=1,2\}\right)$, and Lemma \ref{comp}, we have $c(j,mk) \leq c(\ov{j-mk},\ov{-mk})=t+c(j,mk)-c(j,\ov{-mk})$. Thus we have $c(j,\ov{-mk}) \leq t$ and therefore $j \in  Y_{a}$ and $\phi_{a}(j)=i$. 

2) From Lemma \ref{comp} we have $c(j,m_{j}k)+c(\ov{j+m_{j}k},\ov{-m_{j}k})=2t$ and $c(j,m_{j}k)+c(\ov{j+m_{j}k},m_{j}k)=t+c(j,\ov{-m_{j}k})$. As $c(j,m_{j}k)=a<t$, Lemma \ref{comp} implies $c(\ov{j+m_{j}k},\ov{-m_{j}k})=2t-a>t$.
As $j \in  Y_{a}$, we have $c(j,\ov{-m_{j}k})\geq c(j,m_{j}k)$ and thus $c(\ov{j+m_{j}k},\ov{-m_{j}k})=t+c(j,\ov{-m_{j}k})-c(j,m_{j}k)\geq t$. Thus $\ov{j+m_{j}k} \not\in  \phi_{a}^{-1}(i)$. We have $\phi_{a}^{-1}(i) \subset \{j,\ov{j-mk}\}$ 

As $j \in Y_{a}$, we have that $a=c(j,m_{j}K)\leq c(j,\ov{-m_{j}k}) \leq t$. Now, application of Lemma \ref{comp} yields $c(\ov{j-m_{j}k},m_{j}k)=2t-c(j,mk) \geq t$. By Lemma \ref{comp}, we have $c(\ov{j-m_{j}k},\ov{-m_{j}k})+c(j,\ov{-m_{j}k})= t+c(j,m_{j}k)$, and thus $c(\ov{j-m_{j}k},\ov{-m_{j}k})=a$ if and only if $c(j,\ov{-m_{j}k})=t$. Therefore $\ov{j-m_{j}k} \in Y_{a}$ is and only if $c(j,\ov{-m_{j}k})=t$. Thus (2) is true. 
 
3) From (1) and (2), we have $|R_{a}|= \sum\limits_{i \in Y_{a}}\frac{i}{|\phi^{-1}_{a}(i)|}=|Y_{a}|-\frac{1}{2}|\{j|\ g(j)=\{a,t\}\}|$. Given $j \in \{j|\ g(j)=\{a,t\}\}$, let $m_{j} \in \{1,2\}$ such that $c(j,m_{j}k)=a$. We then have $c(j,\ov{-m_{j}k})=t$. By Lemma \ref{comp}, we have $c(j,m_{j}k)+t=c(j,\ov{-m_{j}k})+c(\ov{j-m_{j}k},\ov{-m_{j}k})$ and thus $c(\ov{j-m_{j}k},\ov{-m_{j}k})=c(j,m_{j}k)=a$. Using  $c(j,mk)=c(\ov{-j-mk},mk)$ and $c(\ov{-j-mk},\ov{-mk})=c(\ov{j-mk},\ov{-mk})$ from Lemma \ref{comp}, we have $c(\ov{-j-m_{j}k},m_{j}k)=c(j,m_{j}k)=c(\ov{j-m_{j}k},\ov{-m_{j}k})=c(\ov{-j-m_{j}k},\ov{-m_{j}k})$. Thus if $j \in \{j|\ g(j)=\{a,t\}\}$, then $\ov{-j-m_{j}k} \in \{j|\ g(j)=\{a\}\}$. Now the map $\lambda: \{j| \ g(j)=\{a,t\}\} \to \{j|\ g(j)=\{a\}\}$ defined by $\lambda(j)=\ov{-j-m_{j}k}$ is well-defined. For $j \in  \{j|\ g(j)=\{a\}\}$, we have $\lambda^{-1}(j)\subset \{\ov{-j}, \ov{-j+k},\ov{-j+2k}\}$. Given $m \in \{1,2\}$, Lemma \ref{comp} gives us $c(\ov{-j},mk)=c(\ov{j-mk},mk)=2t-c(j,-mk)=2t-a$, $c(\ov{-j+mk},mk)=c(\ov{j+mk},mk)=t+c(j,\ov{-mk})-c(j,mk)=t+a-a=t$, and
 $c(\ov{-j+mk},\ov{-mk})=c(j,\ov{-mk})=a$. Thus the map $\lambda$ is a $2$ to $1$ map and therefore $\frac{1}{2}|\{j|\ g(j)=\{a,t\}\}|=|\{j|\ g(j)=\{a\}\}|$.  
\end{proof}
\begin{corollary}\label{rank}
$e_{0}= \left(\dfrac{p+1}{3}\right)^{2t}(2^{t+1}-2)$.
\end{corollary}
\begin{proof}
Lemma \ref{trs} and Theorem \ref{m} imply $e_{0}= |R_{0}|+2=|Y_{0}|-\{j| g(j)=\{0\}\}+2$. We recall that $k=\sum\limits_{m=0}^{t-1}\left(\frac{p-2}{3}\right)p^{2m}+ \left(\frac{2p-1}{3}\right) p^{2m+1}$ and $2k=\sum\limits_{m=0}^{t-1}\left(\frac{2p-1}{3}\right) p^{2m}+ \left(\frac{p-2}{3}\right)p^{2m+1}$. Therefore set $\{j| (c(j,k),c(j,2k))=(0,b)\}$ is made up of numbers of the form $j=\sum\limits_{m=0}^{t-1}a_{2m}p^{2m}+a_{m+1} p^{2m+1}$ satisfying: i) $0\leq a_{2m} < \frac{p+1}{3}$, ii) $0\leq a_{2m+1} < \frac{2(p+1)}{3}$ and, iii) $j \notin \{k,2k\}$. Thus this set has size $2^{t}\left(\dfrac{p+1}{3}\right)^{2t}-2$. Similar arguments yeild $|\{j| j \neq 0\ \text{and} \ (c(j,k),c(j,2k))=(b,0)\}|=2^{t}\left(\dfrac{p+1}{3}\right)^{2t}-2$ and $|\{j| j \neq 0\ \text{and}\ g(j)=\{0\}\}|=\left(\frac{p+1}{3}\right)^{2t}-1$. The result now follows by the principle of inclusion-exclusion.     
\end{proof}

For $0<a<t$, Lemma \ref{trs} shows that $e_{a}=|R_{a}|$.
We will use the transfer matrix method to compute $|R_{a}|$. We construct a weighted digraph $D$, and change the problem of computing $e_{a}$ to that of counting closed walks in $D$ of certain length and weight. 

Let $D$ be a digraph with vertex set $V$, edge set $E$, and with a weight function $wt:E \to \mathfrak{R}$ with values in some commutative ring $\mathfrak{R}$. 
By $M$, we denote the adjacency matrix of $D$ with respect to the weight $wt$.
Given $n\in \bb{Z}_{>0}$, let $C(n)= \sum _{\psi} wt(\psi)$, where the sum is over closed walks in $D$ of length $n$.
The following Lemma which is Corollary $4.7.3$ of \cite{EC1} gives us the generating function $\sum\limits_{n\geq 1}C(n)z^{n}$.
\begin{lemma}\label{trans}
Let  $T(z)= det (I-zM)$, then $\sum\limits_{n\geq 1}C(n)z^{n}= -\dfrac{z T'(z)}{T(z)}$.
\end{lemma}

Consider $A_{1}= \{(\alpha, \gamma, \delta)| \ (\alpha, \gamma, \delta) \in \{0,1,\ldots,p-1\} \times \{1,2\} \times \{1,2\} \}$ and $A_{2}= \{[\alpha, \gamma, \delta] | \ (\alpha, \gamma, \delta) \in \{0,1,\ldots,p-1\} \times \{1,2\} \times \{1,2\} \}$. We construct a bipartite digraph $D=(A_{1}\cup A_{2},E)$. There is an arc $\mrm{e} \in E$ from $(\alpha, \gamma, \delta) \in A_{1}$ to $[\alpha^{'}, \gamma^{'}, \delta^{'}] \in  A_{2}$ if an only if 
\begin{align*}
\alpha+ \dfrac{2p-1}{3} + \gamma &= \beta +p \gamma^{'} \\
& \text{and} \\
\alpha+ \dfrac{p-2}{3} + \delta &= \epsilon +p \delta'
\end{align*}   
for some $\beta, \epsilon \in \{0,1,\ldots,p-1\}$. There is an arc $\mrm{e}_{\lambda} \in E$ from $[\alpha, \gamma, \delta] \in A_{2}$ to $( \alpha', \gamma', \delta ') \in  A_{1}$ if and only if
\begin{align*}
\alpha+ \dfrac{p-2}{3} + \gamma &= \beta +p \gamma^{'} \\
& \text{and} \\
\alpha+ \dfrac{2p-1}{3} + \delta &= \epsilon +p \delta'
\end{align*}   
for some $\beta, \epsilon \in \{0,1,\ldots,p-1\}$. 
The arcs in $D$ of type $\mrm{e}$ and $\mrm{e}_{\lambda}$ are assigned label $\alpha$ and weights $wt(\mrm{e})=wt(\mrm{e}_{\lambda})=x^{\gamma'}y^{\delta'}$. So we have a weight function $wt:E \to  \bb{C}[x,y]$ on $D$. The weight of a walk on $D$ will be the products of the weights of its arcs. 

Given $a,b \in \{0,1,2,\ldots ,2t+1\}$, let $E_{(a,b)}$ be the set of closed walks of length $2t$ and weight $x^{a}y^{b}$. A closed walk of length $2t$ with its initial vertex in $A_{1}$ is said to be of type $A_{1}$, and is of type $A_{2}$ otherwise. Let $Y_{a,b}=\{j \in \{1,2,\ldots,q-2\}\setminus \{k,2k\}|\ g(j)=\{a,b\}\}$. Let $a_{0},a_{1}, \ldots a_{2t}$ be the labels of arcs of a walk $w \in \cup E_{(a,b)}$, then define $\psi(w)= \sum a_{i}p^{i}$.  When $\{a,b\}\cap \{0,\ 2t\}= \emptyset$, we have $\psi(E_{(a,b)}) \subset  Y_{a,b}$. By the $p$-ary add-with-carry-algorithm described in Theorem 4.1 of \cite{X2}, given $j \in Y_{a,b}$, there exist \emph{carry sequences} $(\gamma_{0},\gamma_{1}, \ldots \gamma_{2t-1})$ and $(\delta_{0},\delta_{1}, \ldots \delta_{2t-1})$ with $\gamma_{i}, \delta_{i} \in \{1,2\}$ such that 
\begin{align*}
a_{i}+ \frac{2p-1}{3}+\gamma_{i}= b_{i} + \gamma_{i+1}p \ & & a_{i}+ \frac{p-2}{3}+\delta_{i}= d_{i} + \delta_{i+1}p \ ,\ \text{for even $i$ and;} \\
a_{i}+ \frac{p-2}{3}+\gamma_{i}= b_{i} + \gamma_{i+1}p \ &  & a_{i}+ \frac{2p-1}{3}+\delta_{i}= d_{i} + \delta_{i+1}p \ ,\  \text{for odd $i$}.
\end{align*} 
Here $j= \sum a_{i}p^{i}$, $j+k= \sum b_{i}p^{i}$ and $j+2k=d_{i}p^{i}$.
We can now see that there are exactly two closed walks, one of each type which map to $j$ under $\psi$. If $w(j,A_{1})$ (respectively $w(j,A_{2})$) is the walk of type $A_{1}$ (respectively type $A_{2}$) such that $\psi(w(j,A_{1}))=j$ (respectively $\psi(w(j,A_{2}))=j$), then $wt(w(j,A_{1}))= x^{c(j,k)}y^{c(j,2k)}$ (respectively $wt(w(j,A_{2}))= x^{c(j,2k)}y^{c(j,k)}$). We can now conclude that for $a \neq b$, the restriction of $\psi$ is a bijection from $E_{(a,b)}$ to $Y_{a,b}$. Applying Lemma \ref{trs} 3) gives us
\begin{equation}\label{count}
e_{a}= \sum\limits_{b=a+1}^{t} |E_{(a,b)}|,
\end{equation}
 for all $0<a<t$.
 
 We observe that for all $\alpha, \alpha' \in \{0,1,\ldots,p-1\}$ and $\gamma, \delta \in \{1,2\} \times \{1,2\}$, there is no arc from $(\alpha,\gamma,\delta)$ (resp. $[ \alpha,\gamma,\delta]$) to $[ \alpha',0,1]$ (resp. $(\alpha',1,0)$). 
We may also conclude that
\begin{enumerate}
\item there is an edge from $(\alpha, \gamma, \delta)$ to  $[ \alpha^{'},0,0]$ if and only if $0\leq \alpha< \frac{p+1}{3}-\gamma$;
\item there is an edge from $(\alpha, \gamma, \delta)$ to $[ \alpha^{'},1,0]$ if and only if $\frac{p+1}{3}-\gamma\leq \alpha< \frac{2(p+1)}{3}-\delta $;
\item there is an edge from $(\alpha, \gamma, \delta)$ to $[ \alpha^{'},1,1]$ if and only if $\frac{2(p+1)}{3}-\delta \leq \alpha <p$;
\item there is an edge from $[ \alpha, \gamma, \delta]$ to $(\alpha^{'},0,0)$ if and only if $0 \leq \alpha< \frac{p+1}{3}-\delta$;
\item there is an edge from $[\alpha, \gamma, \delta]$ to $( \alpha^{'},0,1)$ if and only if $\frac{p+1}{3}-\delta \leq \alpha< \frac{2(p+1)}{3}-\gamma $;
\item and there is an edge from $[\alpha, \gamma, \delta]$ to $( \alpha^{'},1,1)$ if and only if $\frac{2(p+1)}{3}-\gamma \leq \alpha <p$.  
 
\end{enumerate}

 Let $M$ be the adjacency matrix of the weighted digraph $D$ and let $U$ be the $\bb{C}(x,y)$ vector space generated by the vertex set $A_{1} \cup A_{2}$ (of $D$) as a basis. By abuse of notation, we may assume $M \in \mrm{End}(U)$.

Let
$ h_{1}:=\sum\limits_{(\gamma, \delta)} \sum\limits_{\alpha'< \frac{p+1}{3}-\gamma} [ \alpha', \gamma, \delta]$, 
$h_{2}:=\sum\limits_{(\gamma, \delta)} \sum\limits_{ \frac{p+1}{3}-\gamma \leq \alpha' < \frac{2(p+1)}{3}-\delta} [ \alpha', \gamma, \delta]$, 
$h_{3}:= \sum\limits_{(\gamma, \delta)} \sum\limits_{\alpha'\geq \frac{2(p+1)}{3}-\delta} [ \alpha', \gamma, \delta]$, 

$h'_{1}:=\sum\limits_{(\gamma, \delta)} \sum\limits_{\alpha'< \frac{p+1}{3}-\delta} ( \alpha',\ \gamma,\ \delta)$, 
$h'_{2}=\sum\limits_{(\gamma, \delta)} \sum\limits_{ \frac{p+1}{3}-\delta \leq \alpha' < \frac{2(p+1)}{3}-\gamma} ( \alpha', \gamma, \delta)$, 
and $h'_{3}:=\sum\limits_{(\gamma, \delta)} \sum\limits_{\alpha'\geq \frac{2(p+1)}{3}-\gamma} ( \alpha', \gamma, \delta)$. 

We can see that $M(A_{1} \cup A_{2})=\{h_{1},h_{2},h_{3},h'_{1},h'_{2},h'_{3}\}$. We also have,

\begin{align*}
M(h'_{1})= \frac{p+1}{3} h_{1} + \frac{p+1}{3} x h_{2} + \frac{p-2}{3}xy h_{3},\\
M(h'_{2})= \frac{p+1}{3} h_{1} + \frac{p-2}{3} x h_{2} + \frac{p+1}{3}xy h_{3}, \\
M(h'_{3})= \frac{p-2}{3} h_{1} + \frac{p+1}{3} x h_{2} + \frac{p+1}{3}xy h_{3}, \\ 
M(h_{1})= \frac{p+1}{3} h'_{1} + \frac{p+1}{3} y h'_{2} + \frac{p-2}{3}xy h'_{3}, \\
M(h_{2})= \frac{p+1}{3} h'_{1} + \frac{p-2}{3} y h'_{2} + \frac{p+1}{3}xy h'_{3},\ \text{and} \\
M(h_{3})= \frac{p-2}{3} h'_{1} + \frac{p+1}{3} y h'_{2} + \frac{p+1}{3}xy h'_{3}.
\end{align*}

Let $W$ be the subspace of $U$ generated by $\{h_{i},\ h'_{i}| \ i=1,2,3\}$, then $M(U)=W$. The set $\beta=\{h_{i},\ h'_{i}| \ i=1,2,3\} $ is linearly independent, and thus is a basis for $W$. Let $M_{[\beta]}$ be the matrix representation of $M|_{W}$ with respect to the basis $\beta$ of $W$. From above we see that $M_{[\beta]}$ is

$$\left[\begin{array}{c c c c c c}
0 & 0 & 0 & \frac{p+1}{3} & \frac{p+1}{3} & \frac{p-2}{3} \\
0 & 0 & 0 & \frac{p+1}{3}y & \frac{p-2}{3}y & \frac{p+1}{3}y \\
0 & 0 & 0 & \frac{p-2}{y}xy & \frac{p+1}{3}xy & \frac{p+1}{3}xy \\
\frac{p+1}{3} & \frac{p+1}{3} & \frac{p-2}{3} & 0 & 0 & 0 \\
\frac{p+1}{3}x & \frac{p-2}{3}x & \frac{p+1}{3}x &0 &0 & 0 \\
\frac{p-2}{y}xy & \frac{p+1}{3}xy & \frac{p+1}{3}xy &0 &0 &0

\end{array} \right].$$

As $det(M_{[\beta]})= -p^{2}x^{3}y^{3} \neq 0$, we have $W \cap \ker(M)=\{0\}$ and thus $U=\ker(M)\oplus W$.

Thus the characteristic polynomial of $M$ is $f(z)=z^{8p-6}\mrm{det}(zI-M_{[\beta]})$.
Careful computation shows that
$ det(zI-M_{[\beta]})=z^{6}- Pz^{4} + Qz^{2}-R,$

 where $P=\left( \left(\frac{p+1}{3}\right)^{2}(x^{2}y^{2}+x^{2}y+xy^{2}+x+y+1)+\left(\frac{p-2}{3}\right)^{2}3xy\right)$,
 
 $Q=\left( \left(\frac{p+1}{3}\right)^{2}(xy)(x^{2}y^{2}+x^{2}y+xy^{2}+x+y+1)+\left(\frac{2p-1}{3}\right)^{2}3x^{2}y^{2}\right)$, and $R=p^{2}x^{3}y^{3}.$ Thus $f(z)= z^{8p}-Pz^{8p-2}+Qz^{8p-4}-Rz^{8p-6}$.

Let $C(n)= \sum _{\psi} wt(\psi)$, where the sum is over closed walks in $D$ of length $n$. As $D$ is a bipartite graph, we have $C(n)=0$ for all odd $n$. By Lemma \ref{trans}, we have

$$\sum_{t \geq 1} C(2t)z^{2t}= -\frac{z T'(z)}{T(z)},$$
where $T(z)= det (I-zM)$. The characteristic polynomial of $M$ was computed above to be $z^{8p}-Pz^{8p-2}+Qz^{8p-4}-Rz^{8p-6}$, and thus we have

$$\sum_{t \geq 1} C(2t)z^{2t}= \frac{2Pz^{2}-4Qz^{4}+6Rz^{6}}{1-(Pz^{2}-Qz^{4}+Rz^{6})}.$$  

Let $C(2t)=0$ for $t \leq 0$. We have $\sum\limits_{t \geq 1}(C(2t)-PC(2t-2)+QC(2t-4)- RC(2t-6))z^{t}= 2Pz- 4Qz^{2}+6Rz^{3}$. Thus we have

\begin{align*}
C(2)=2P \\ C(4)= 2(P^{2}-2Q),\\
 C(6)= 6R+2(P^{3}-2QP)-2PQ, \\
\text{and}\ C(2t)= PC(2t-2)-QC(2t-4) + R C(2t-6) &\ \text{for}\ t> 3.
\end{align*} 

The coefficient of $x^{a}y^{b}$ in $C(2t)$ is $|E_{(a,b)}|$. Given $a<t$, we have from \eqref{count} that $e_{a}= \sum\limits_{a<b\leq t} |E_{(a,b)}|$.
Application of Theorem \ref{m} and Corollary \ref{rank} yields Theorem \ref{e3}.

\section*{Acknowledgement}
I thank Prof. Peter Sin for his valuable suggestions and feedback. I am grateful to Prof. David Saunders for providing computational data. I am indebted to the anonymous referees for several helpful suggestions.  
  
\bibliographystyle{plain}
\bibliography{my}    
\end{document}